\documentclass[12pt]{amsart}
\usepackage{xypic}
\usepackage{amssymb}
\usepackage{amsthm}
\usepackage{bm}
\usepackage{latexsym}
\usepackage{amsmath}
\usepackage{eufrak}
\usepackage{mathrsfs}
\usepackage{amscd}
\usepackage[all]{xy}
\usepackage[usenames]{color}
\usepackage[dvips]{graphicx}
\usepackage{color}
\usepackage{amscd}
\theoremstyle{plain}

\newtheorem{theorem}{Theorem}[section]
\newtheorem{proposition}[theorem]{Proposition}

\newtheorem{corollary}[theorem]{Corollary}

\theoremstyle{definition}

\newcommand{\appsection}[1]{\let\oldthesection\thesection
\renewcommand{\thesection}{Appendix \oldthesection}
\section{#1}\let\thesection\oldthesection}

\newtheorem{definition}[theorem]{Definition}

\theoremstyle{remark}

\newtheorem{remark}[theorem]{Remark}
\newtheorem{example}[theorem]{Example}

\def\ZZ{{\mathbb{Z}}}
\def\QQ{{\mathbb{Q}}}
\def\CC{{\mathbb{C}}}
\def\PP{{\mathbb{P}}}

\def\HH{{\mathcal{H}}}
\def\OO{{\mathcal{O}}}

\def\LL{{\mathcal{L}}}

\def\AR{{\mathcal{A}}}

\newcommand{\HK}{\overline{\mathcal{H}}_n}

\begin{document}
\bibliographystyle{amsplain}
\title[Construcci\'on]{Chern slopes of simply connected complex surfaces of general type are dense in [2,3]}
\author{\textrm{Xavier Roulleau and Giancarlo Urz\'ua}}
\date{\today}
\dedicatory{Dedicated to Igor Dolgachev on the occasion of his 70th birthday}

\email{Xavier.Roulleau@math.univ-poitiers.fr, urzua@mat.puc.cl}



\begin{abstract}
We prove that for any number $r \in [2,3]$, there are spin (resp. non-spin and minimal) simply connected complex surfaces of general type $X$ with $c_1^2(X)/c_2(X)$ arbitrarily close to $r$. In particular, this shows the existence of simply connected surfaces of general type arbitrarily close to the Bogomolov-Miyaoka-Yau line. In addition, we prove that for any $r \in [1,3]$ and any integer $q\geq 0$, there are minimal complex surfaces of general type $X$ with $c_1^2(X)/c_2(X)$ arbitrarily close to $r$, and $\pi_1(X)$ isomorphic to the fundamental group of a compact Riemann surface of genus $q$. %A central ingredient is a new family of special arrangements of elliptic curves in the projective plane.
\end{abstract}

\maketitle

%\begin{center} {\small \textit{Dedicated to Igor Dolgachev on the occasion of his 70th birthday}} \end{center}

%\tableofcontents

%----------------------------------------------------------------------------------------------------------------------------------------------
\section{Introduction} \label{intro}

Let $X$ be a minimal nonsingular projective surface of general type over $\CC$. Its Chern numbers $c_1^2(X), c_2(X)$ satisfy $c_1^2(X)>0$, $c_2(X)>0$, the Noether inequality $\frac{1}{5} c_2(X) - \frac{36}{5} \leq c_1^2(X)$, and the Bogomolov-Miyaoka-Yau inequality \cite{Bog79}, \cite{Miy77,Yau77} $$c_1^2(X) \leq 3 c_2(X),$$ where $c_1^2(X)=3c_2(X)$ if and only if the universal cover of $X$ is the complex two dimensional ball \cite{Yau77,Miy82}. The geography problem \cite{Per87} asks for which pair of integers $(a,b)$ there exists $X$ such that $c_1^2(X)=a$ and $c_2(X)=b$. We refer to \cite[VII \S8]{BHPV04} for general existence results. One of them, due to Sommese \cite{So84}, states that every rational point in $[1/5,3]$ occurs as the slope $c_1^2(X)/c_2(X)$ of some (irregular) surface $X$.

Since the early stages of the geography problem, the question was naturally sharpened by imposing the constraint of simply connectedness. Because of the above result of Yau and Miyaoka, a surface $X$ (of general type) satisfying $c_1^2(X)=3c_2(X)$ is not simply connected. \textit{Is this the only additional restriction?} If we divide surfaces according to their index $\frac{1}{3}(c_1^2(X)-2c_2(X))$, then there are satisfactory results for surfaces with negative index due to Persson \cite{Per81}. For instance, all rational slopes between $1/5$ and $2$ are realized as Chern slopes of simply connected surfaces.

Simply connected surfaces with nonnegative index are more difficult to find. In the late $70$s, it was actually conjectured that any simply connected surface of general type has negative index (Bogomolov Watershed Conjecture, see \cite{V78}, \cite{Per87}). In $1984$, Moishezon and Teicher \cite{MT87} gave the first examples with positive index. After this, there were several attempts to fill out this positive index region. One would at least want to fill out the Chern slope interval $[2,3[$ as much as possible. In particular, \textit{are there simply connected surfaces of general type with Chern slope arbitrarily close to $3$?} This is a well known open problem; see for example \cite{Holz78-79}, \cite{Holz81}, \cite{Per87}, \cite{Ch87}, \cite[\S4]{PP93}, \cite{PPX96}, \cite[VII\S8B]{BHPV04}, \cite{U10}, \cite{U11}. The construction of this type of surfaces becomes harder as we approach $3$. The highest known slope was given in \cite{U10}, proving the existence of such surfaces with $c_1^2/c_2$ arbitrarily close to $\frac{71}{26}\approx 2.730769$. Our main theorem is

\begin{theorem}
For any number $r \in [2,3]$, there are spin (resp. non-spin and minimal) simply connected surfaces of general type $X$ with $c_1^2(X)/c_2(X)$ arbitrarily close to $r$.
\end{theorem}

In particular, this shows that $3$ is indeed the sharp upper bound for slopes of both spin and non-spin simply connected surfaces. Let us recall that a simply connected surface is called \textit{spin} if its canonical class is $2$-divisible. For the topological interest in the existence of spin surfaces we refer to \cite[\S4]{PP93},\cite{PPX96}.

A key ingredient in the construction of the surfaces $X$ is the existence of a family of special arrangements of elliptic curves $\mathcal{H}'_n$ in the blow-up $H$ of $\PP^2$ at the $12$ triple points of the dual Hesse arrangement of lines. These arrangements are related to special arrangements of elliptic curves in $\CC/\ZZ[\zeta] \times \CC/\ZZ[\zeta]$ where $\zeta=e^{2 \pi i/6}$. The latter were discovered by Hirzebruch in \cite{Hirz84}. They are connected to open ball quotients (see \cite{KH05} for more on this type of special arrangements). The arrangements $\mathcal{H}'_n$ in $H$ have many simple $4$-fold points as singularities (we later call them $4$-points), and a particular divisibility property in Pic$(H)$. This allows us to consider certain cyclic covers \cite{EV92} which are branched along the arrangement, together with other curves, but ``avoid" the exceptional divisors from all the $4$-points. We make use of the more detailed description of cyclic covers developed in \cite{U08,U10}, which in particular makes explicit the use of log invariants, classical Dedekind sums, and lengths of Hirzebruch-Jung continued fractions, to estimate the Chern slope. To demonstrate that these surfaces are actually simply connected, we show that certain blow-ups of the surfaces $X$ admit a fibration to $\PP^1$, which has sections and a simply connected fiber; we then use a method from \cite{X91}. The spin issue is more delicate, among other things it involves a choice of specific multiplicities for the curves in the branch locus.

The construction has many degrees of freedom due to the properties of the key arrangements of elliptic curves. For example, the surfaces we construct in the previous theorem admit many deformations. For this, we add a general arrangement of $d$ lines to the branch locus, where $d$ is essentially independent of the parameters used in the construction. In addition, the construction allows us the use of a specific base change, over the fibration to $\PP^1$ mentioned above, to prove the following.

\begin{corollary}
Let $q>0$ be an integer. Then, for any number $r \in [1,3]$, there are minimal surfaces of general type $X$ with $c_1^2(X)/c_2(X)$ arbitrarily close to $r$, and $\pi_1(X)$ isomorphic to the fundamental group of a compact Riemann surface of genus $q$.
\end{corollary}

The following is an overall outline of the construction of the surfaces $X$ in the above theorem, in connection to the sections of the paper. Let us fix $r \in [2,3]$, and $\epsilon >0$. We construct the surfaces $X(=X_n)$ via morphisms $f \colon X_n \to Y_n$, such that there exists a function $\lambda \colon \QQ_{>0} \to \QQ$ and a rational number $x$ satisfying $|\lambda(x)-r| < \epsilon$ and $\text{lim}_{n \to \infty} \ c_1^2(X_n)/c_2(X_n) = \lambda(x)$. To achieve this, it is crucial to have a special branch divisor $A \subset Y_n$ for the morphisms $f$. The surfaces $Y_n$ are built through blow-ups at specific points of the surface $H$ explained above. The key part of $A$ comes from the arrangements of elliptic curves $\mathcal{H}'_n$ in $H$. Roughly speaking, these arrangements are essentially defined as images of the Hirzebruch's arrangements $\mathcal{H}_n$ in \cite{Hirz84}, by means of a cyclic quotient of a blow-up of $\CC/\ZZ[\zeta] \times \CC/\ZZ[\zeta]$. In \S2 we review the construction of the arrangements $\mathcal{H}_n$. In \S3 we give the quotient construction of $H$, and we define the arrangements $\mathcal{H}'_n$. In \S4 we compute their log Chern numbers. These log invariants are central for our asymptotic results. In \S5 we prove our main result for spin surfaces. There, we also discuss all the morphisms involved, the branch divisors $A$, various divisibilities in Picard groups, formulas for Chern invariants, and arithmetic numbers which are relevant in these formulas. Using particular curves in $A$, we also prove in \S5 that $X$ is simply connected. In \S6, by applying the previous work in a similar set up, we prove our main result for non-spin minimal surfaces. We point out that the functions $\lambda$ have as image either $]1.375,3[$ for spin surfaces or $]1,3[$ for non-spin surfaces, and so we indeed obtain density for the closure of these two intervals. At the end of \S6, we use the non-spin surfaces to prove the above corollary via particular base changes.

%In \S2, we review the construction of the arrangements $\mathcal{H}_n$ of elliptic curves due to Hirzebruch \cite{Hirz84}. In \S3 we present the key arrangements of elliptic curves $\mathcal{H}'_n$. In \S4 we compute their log Chern numbers. These log invariants are central for our results. In \S5 we prove our main result for spin surfaces. In that section we also discuss all the morphisms, divisibilities, arithmetic numbers, and formulas for Chern invariants which are used in the constructions. In \S6 we prove our main result for non-spin minimal surfaces, and the corollary above.

\subsection*{Conventions}
For a nonsingular surface $S$, $K_S$ denotes a canonical divisor of $S$. An arrangement of curves is a collection of curves $\{C_1, \ldots, C_r \}$ on a surface. We loosely consider it as either a set, a divisor $C_1+ \ldots+C_r$, or a curve $\bigcup_{i=1}^r C_i$. A $k$-point of an arrangement is a point of it locally of the form $(0,0) \in \{(x-a_1 y) \cdots (x- a_k y)=0 \} \subset \CC_{x,y}^2$ for some $a_i \neq a_j$. For an invertible sheaf $\LL$, we write $\LL^{r}:= \LL^{\otimes r}$. The genus of a compact Riemann surface $C$ is denoted by $g(C)$. The (topological) fundamental group of $Z$ is denoted by $\pi_1(Z)$.

\subsection*{Acknowledgements}
We would like to thank Igor Dolgachev for helpful discussions, and Fabrizio Catanese for pointing out that Proposition \ref{dualHesse} is a classical fact due to A. Comessatti (see \cite[\S 2]{CC93}). The main results of the present article were found while the first author was visiting Pontificia Universidad Cat\'olica de Chile under an invitation from the second. Xavier Roulleau beneficed from a sabbatical semester d\'el\'egation CNRS. Giancarlo Urz\'ua was supported by the FONDECYT Inicio grant 11110047 funded by the Chilean Government.

%----------------------------------------------------------------------------------------------------------------------------------------------
\section{Hirzebruch elliptic arrangements} \label{hirz}

Let $\zeta=e^{2\pi i/6}$, and let $T$ be the elliptic curve $$T=\CC/\ZZ[\zeta],$$ where $\ZZ[\zeta]=\{ a+ \zeta b \,|\, a,b \in \ZZ \}$ are the Eisenstein integers. In \cite{Hirz84}, Hirzebruch constructs special arrangements of elliptic curves in $T\times T$. We recall his construction, trying not to deviate much from his notation. Write $(z,w)$ for points in $T\times T$. For each $n\in\ZZ[\zeta]$, let us denote by $U_{n}$ the group of $n$-torsion points $$ U_{n}=\{(z,w)\in T\times T \,|\, nz=nw=0\}.$$ The group $U_{n}$ has order $(n\bar{n})^{2}$, where $\bar{n}$ denotes the complex conjugation of $n$. (In \cite{Hirz84}, $n$ is an integer.)

Let us consider the following four elliptic curves $$T_{0}=\{z=0\},\, T_{1}=\{z=w\},\, T_{\infty}=\{w=0\},\, T_{\zeta}=\{z=\zeta w\}.$$ These four curves pass trough $(0,0)$, and do \textit{not} intersect anywhere else. The group $U_{n}$ acts on $T\times T$ by translation. The stabilizer of each of the
four curves $T_{i}$ in $U_n$ has order $n\bar{n}$. Therefore the orbit $D_{i}=U_{n}(T_{i})$
of $T_{i}$ ($i=0,1,\zeta,\infty$) consists of $n\bar{n}$ disjoint nonsingular
elliptic curves, all being fibers of the fibrations $\pi_{i} \colon T\times T\to T$ defined as $$\pi_{0}(z,w)=z,  \ \pi_1(z,w)=z-w, \ \pi_{\infty}(z,w)=w, \ \pi_{\zeta}(z,w)=z-\zeta w.$$

The arrangement $$\mathcal{H}_{n}:= D_{0}+D_{1}+D_{\zeta}+D_{\infty}$$ has $4n\bar{n}$ elliptic curves. The singularities of $\mathcal{H}_{n}$ are precisely $(n\bar{n})^{2}$ $4$-points.

For example, let us take $U_{1-\zeta^{2}}$, which is the group of smallest order among all non-trivial groups $U_{n}$ with $n\in\ZZ [\zeta]$. The arrangement $\mathcal{H}_{1-\zeta^{2}}$ consists of $12$ elliptic curves, each containing $3$ points of $U_{1-\zeta^{2}}$, and nine $4$-points; it gives a Hesse type configuration $(12_{3},9_{4})$.

For $n\in\ZZ$, Holzapfel proves in \cite{Holz86} that the complement in the blow-up at the $n^4$ $4$-points of the strict transform of the $\mathcal{H}_{n}$ is a ball quotient by an arithmetic group. Holzapfel also considered other similar special elliptic arrangements, for example using the Gaussian integers; cf. \cite{KH05}. They would also work in our upcoming constructions.

%----------------------------------------------------------------------------------------------------------------------------------------------
\section{Dual Hesse and elliptic arrangements} \label{arrang}

The endomorphism $[\zeta^{2}]$ of $T\times T$ given by $$[\zeta^{2}](z,w)=(\zeta^{2}z,\zeta^{2}w)$$ (where $\zeta=e^{2i\pi/6}$) is an order $3$ automorphism of $T\times T$. Its fixed point set
is the group of $9$ elements $U_{1-\zeta^{2}}\subset U_{3}$.

Let $\sigma \colon B \to T \times T$ be the blow-up at the $9$ points in $U_{1-\zeta^2}$.
The strict transform of the $T_{i}$ is denoted again by $T_i$. Thus
$T_{i}^{2}=-3$ in $B$. The automorphism $[\zeta^2]$ acts on $B$, and
it is the identity on each of the $9$ exceptional curves of $\sigma$. Let $H=B/ \langle [\zeta^2] \rangle$, and let $\pi \colon B \to H$
be the quotient map. For each $i$, we have the commutative diagram
$$ \xymatrix{ B \ar[r]^{\pi} \ar[d]^{\pi_i \circ \sigma}  & H \ar[d]^{\pi'_i} \\
   T \ar[r] & \PP^1} $$ where the horizontal arrows are quotients by $\ZZ / 3 \ZZ$, and the $\pi'_i$ is the induced elliptic fibration on $H$.

\begin{proposition}
The surface $H$ is the blow-up $\tau \colon H \to \PP^2$ at the twelve $3$-points of the dual Hesse arrangement $$(x^3 - y^3)(x^3 - z^3)(y^3 - z^3)=0.$$ The image of the strict transform of $\mathcal{H}_{1-\zeta^{2}}$ by $\sigma$ is the exceptional divisor of $\tau$. The image of the nine $(-1)$-curves in $B$ is the strict transform under $\tau$ of the nine lines in the dual Hesse arrangement.
\label{dualHesse}
\end{proposition}

\begin{proof}
For any $i$, the elliptic fibration $\pi'_i \colon H \to \PP^1$ has $3$ trees of $4$ $\PP^1$'s and nine sections from the image of the curves in $\mathcal{H}_{1-\zeta^{2}}$. This implies that $H$ is simply connected; cf. \cite{X91}. Moreover, if $C$ is the exceptional divisor of $\sigma$, then $\pi^*(K_H) + 2C=K_B=C$, and so $\pi^*(K_H)=-C$ and $K_H^2=-3$. If $e(X)$ is the topological Euler characteristic of $X$, then we also have $e(H)=\frac{1}{3}e(B)+\frac{2}{3}e(C)=15$. The image of $C$ under $\pi$ are nine disjoint $(-3)$-curves, and the image of the twelve elliptic curves in the strict transform of $\mathcal{H}_{1-\zeta^{2}}$ is a divisor in $H$ formed by twelve disjoint $(-1)$-curves. After we blow-down these $12$ curves, we obtain $\PP^2$, and the nine $(-3)$-curves become an arrangement of $9$ lines with twelve $3$-points. It is well-known that, up to a projective equivalence, this is the dual Hesse arrangement; for example see \cite[Thm.7.2]{U10}.
\end{proof}

We now define the arrangements of elliptic curves in $H$ which are going to be key in the next sections. Let us first to come back to $T \times T$ and the automorphism $[\zeta^2]$. For each $i \in \{0,1,\zeta,\infty \}$, the curve $T_{i}$ is stabilized by $[\zeta^{2}]$, and contains $3$ fixed points including $0$. These points are in $U_{1-\zeta^{2}}$, therefore the orbit $U_{1-\zeta^{2}}(T_{i})$ consists of $3$ elliptic curves. For $n \in \ZZ[\zeta]$, we recall that $D_{i}$ is a union of $n\bar{n}$ disjoint elliptic curves which are fibers of $\pi_{i}$. Let us denote again by $D_{i}$ the strict transform in $B$ of $D_{i}$. We have two types of elliptic arrangements in $H$ according to the $(1-\zeta^2)$-divisibility of $n$.

\textit{Suppose that $1-\zeta^{2}$ divides $n$}. In this paper we consider only this case. For each $i\in\{0,1,\zeta,\infty\}$, it is immediate to verify that $[\zeta^{2}]$ acts on the divisor $D_{i}$, stabilizing the three curves in the orbit $U_{1-\zeta^{2}}(T_{i})$ and permuting the other $n\bar{n}-3$ curves.

Let us consider the situation in $B$. Each curve in $U_{1-\zeta^{2}}(T_{i})$ cuts $3$ $(-1)$-curves, the other components of $D_{i}$ are disjoint
from $(-1)$-curves. As we have said, the $(-1)$-curves are pointwise fixed
by $[\zeta^{2}]$. The image of $U_{1-\zeta^{2}}(T_{i})$ under $\pi$
is formed by three $(-1)$-curves. The other $(n\bar{n}-3)/3$ disjoint elliptic
curves in the image of $D_{i}$ are fibers of $\pi'_{i}$. We denote
them by $\mathcal{E}_{i}$. Thus we define in $H$ the following arrangement $$\mathcal{H}'_n := \mathcal{E}_{0} + \mathcal{E}_{1} + \mathcal{E}_{\infty} + \mathcal{E}_{\zeta}.$$

On $T\times T$, the arrangement $\mathcal{H}_{n}$ has $(n\bar{n})^{2}$
$4$-points, which are the $n$-torsion points. Since $U_{1-\zeta^{2}}\cap U_{n}=U_{1-\zeta^{2}}$,
the arrangement in $B$ contains $((n\bar{n})^{2}-9)$ $4$-points,
none of them lying on the ramification locus of $\pi$. Therefore, $\mathcal{H}'_{n}$
consists of $4(n\bar{n}-3)/3$ elliptic curves, and its singularities
are precisely $4(n\bar{n}-3)$ $3$-points and $(n\bar{n}-3)(n\bar{n}-9)/3$
$4$-points. Notice that for any $i\neq j$, if $\Gamma_{i}$ and $\Gamma_{j}$
are elliptic curves in $\mathcal{E}_{i}$ and $\mathcal{E}_{j}$ respectively,
then they intersect transversally at $3$ points.

%----------------------------------------------------------------------------------------------------------------------------------------------
\section{Log Chern numbers} \label{log}

In this section we recall some basic definitions and facts around log differentials; cf. \cite{EV92}, \cite[\S2]{U10}. We then compute the log Chern invariants for the arrangements $\mathcal{H}'_n$, and for certain modifications $\HK$. This computation shows the importance of $\HK$ for the upcoming constructions.

Let $Y$ be a nonsingular projective surface. Let $A$ be a simple normal crossing divisor in $Y$. This means, $A$ is an arrangement of nonsingular curves in $Y$, and its singularities are only nodes (i.e. $2$-points). The sheaf of log differentials along $A$, denoted by $\Omega_Y^1(\log A)$, is the $\OO_Y$-submodule of $\Omega_Y^1 \otimes \OO_Y(A)$ satisfying
\begin{itemize}
\item[(1)] $\Omega_Y^1(\log A)|_{Y\setminus A}=\Omega_{Y\setminus A}^1$. \item[(2)] At any point $P$ in $A$, we have $\omega_P \in
\Omega_Y^1(\log A)_P$ if and only if $\omega_P = \sum_{i=1}^m a_i \frac{dy_i}{y_i} + \sum_{j=m+1}^{2} b_j dy_j$, where $(y_1,y_2)$ is a local
system around $P$ for $Y$, and $\{ y_1\cdots y_m=0 \}$ defines $A$ around $P$ (and so it is either $\{y_1=0\}$ or $\{y_1 y_2=0 \}$).
\end{itemize}

Hence, $\Omega_Y^1(\log A)$ is a locally free sheaf of rank two. Notice that $\bigwedge^2 \Omega_Y^1(\log A) \simeq \OO_Y(K_Y + A)$. In analogy to the Chern invariants of a nonsingular projective surface, the log Chern invariants of the pair $(Y,A)$ are defined as $\bar{c}_i(Y,A):= c_i({\Omega_Y^1(\log A)}^{\vee})$ for $i= 1,2$, where ${\Omega_Y^1(\log A)}^{\vee}$ is the dual sheaf of $\Omega_Y^1(\log A)$.

The log Chern numbers of $(Y,A)$ are \begin{center} $\bar{c}_1^2(Y,A):= c_1 \big({\Omega_Y^1(\log A)}^{\vee} \big)^2 \ \ \ \
\text{and} \ \ \ \ \bar{c}_2(Y,A):= c_2 \big({\Omega_Y^1(\log A)}^{\vee} \big).$ \end{center}Hence for $A=0$ we recover the Chern numbers of $Y$, which are denoted by $c_1^2(Y)$ and $c_2(Y)$. Sometimes we may drop $Y$ or $(Y,A)$ if the context is understood.

An arrangement of nonsingular curves $\AR$ in a nonsingular projective surface $Z$ is called \textit{simple crossings} if it has only $k$-points as singularities. The model example are line arrangements in $\PP^2$. The log Chern numbers of the pair $(Z,\AR)$ are defined as the log Chern numbers of the pair $(Y,A)$, where $Y \to Z$ is the blow-up of $Z$ at all $k$-points of $\AR$ with $k>2$, and $A$ is the (reduced) total transform of $\AR$.

\begin{proposition} \cite[Prop.4.6]{U10}
Let $\AR= \{C_1,\ldots, C_d \}$ be a simple crossing arrangement in $Z$. Let $t_k$ be the number of $k$-points in $\AR$. Then, $$\bar{c}_1^2(Z,\AR) = c_1^2(Z) - \sum_{i=1}^d C_i^2 + \sum_{k\geq
2} (3k-4)t_k + 4 \sum_{i=1}^d (g(C_i)-1),$$ and $\bar{c}_2(Z,\AR) = c_2(Z) + \sum_{k\geq 2} (k-1)t_k + 2 \sum_{i=1}^d (g(C_i)-1)$.
\label{logchern}
\end{proposition}

Let $\varphi_n \colon Z_n \to H$ be the blow-up at all the $4$-points of ${\HH}'_n$. Let $\HK$ be the strict transform of ${\HH}'_n$ under $\varphi_n$. Using the formulas in Proposition \ref{logchern}, we have $$\bar{c}_1^2(H,{\HH}'_n)= \frac{8}{3}(n \bar{n})^2 - 12 n \bar{n} +9, \ \ \bar{c}_2(H,{\HH}'_n)=  (n \bar{n})^2 - 4 n \bar{n} + 18,$$ and $$\bar{c}_1^2(Z_n,{\HK})= (n \bar{n})^2+8n \bar{n}- 36, \ \ \bar{c}_2(Z_n,{\HK})=  \frac{1}{3}n \bar{n}(n \bar{n}+12).$$

At this point we highlight that the elimination of all $4$-points from ${\HH}'_n$ make a big asymptotical difference with respect to log invariants.

%----------------------------------------------------------------------------------------------------------------------------------------------
\section{Density for spin simply connected surfaces} \label{densespin}

Recall that for each $i \in \{0,1,\infty,\zeta \}$ we have the commutative diagram $$ \xymatrix{ T \times T \ar[d]^{\pi_i} & B  \ar[l]_{ \ \ \ \ \sigma} \ar[r]^{\pi} & H \ar[dl]_{{\pi'}_i} \ar[d]^{\tau} & Z_n \ar[l]_{ \ \ \varphi_n} \\ T \ar[r] & \PP^1 & \PP^2 & }$$

The three singular fibers of $\pi'_i$ are denoted by $F_{i,1}, F_{i,2}, F_{i,3}$. Each $F_{i,j}$ consists of four $\PP^1$'s: one central curve $N_{i,j}$ with multiplicity $3$, and three reduced curves transversal to $N_{i,j}$ at one point each. We write $N_i=N_{i,1}+N_{i,2}+N_{i,3}$. Let $M$ be the $9$ $\PP^{1}$'s from the lines of the dual Hesse arrangement, and let $N$ be the $12$ exceptional $\PP^{1}$'s from its twelve $3$-points. We have $N=\sum_{i=0,1,\zeta,\infty}N_{i}$, and $$F_{i,1}+F_{i,2}+F_{i,3}=M+3N_{i}.$$

Let $\alpha>0, \beta > 0$ be integers. Let us take $n=12\alpha p$, where $p \geq 5$ is a prime number. Let $\mathcal{E}'_i$ be $8 \beta^2 p^2$ general fibers of $\pi'_i$.  Let $d>0$ be an integer, and let $\AR_{8d}=\sum_{i=1}^{8d} L_i$ be the pre-image under $\tau$ of an arrangement of $8d$ general lines in $\PP^2$ (so it does not contain any singularity of ${\HH}'_n + \sum_{i=0,1,\zeta,\infty} \mathcal{E}'_i$, and it is transversal to all of its curves).

For each $i$ we have that $\mathcal{E}_i \sim \frac{n^2-3}{3} F_i$, where $F_i$ is the class of a fiber of ${\pi'}_i$, and so $$3\mathcal{E}_i + F_{i,1}+F_{i,2}+F_{i,3} \sim  n^2 F_i.$$ We define $a_0=a_1=b_i=1$ for $1\leq i \leq 4d$, and $a_{\infty}=a_{\zeta}=b_i=2p-1$ for $4d+1 \leq i \leq 8d$. Let $L$ be the pull-back of the class of a line. Then $$ \OO_H \Big( \sum_{i=0,1,\zeta,\infty} 3 a_i \mathcal{E}_i + \sum_{i=0,1,\zeta,\infty} 3 a_i \mathcal{E}'_i + \sum_{i=0,1,\zeta,\infty} a_i (F_{i,1}+F_{i,2}+F_{i,3})+ \sum_{i=1}^{8d} 3 b_i L_i \Big)$$ is isomorphic to $\LL_0^{4p}$ where $$\LL_0:= \OO_H \Big( 6p(6\alpha^2+\beta^2)\big(\sum_{i=0,1,\zeta,\infty} a_i F_i \big)+6dL \Big).$$

For each $i$, we denote by the same symbol the strict transform  in $Z_n$ of the divisors $\mathcal{E}_i$, $\mathcal{E}'_i$, $L_j$, $F_{i,j}$.
% by $\mathcal{E}_i$, $\mathcal{E}'_i$, $L_j$, $F_{i,j}$.
Then $$\OO_{Z_n} \big( \sum_{i=0,1,\zeta,\infty} 3 a_i \mathcal{E}_i + \sum_{i=0,1,\zeta,\infty} 3 a_i \mathcal{E}'_i + \sum_{i=0,1,\zeta,\infty} a_i (F_{i,1}+F_{i,2}+F_{i,3})+ \sum_{i=1}^{8d} 3 b_i L_i  \big)$$ is $\LL_1^{4p}$ where  $\LL_1:= \varphi_n^*(\LL_0) \otimes \OO_{Z_n}(-3E)$, and $E$ is the exceptional divisor of $\varphi_n$.

Consider the blow-up $\sigma_n \colon Y_n \to Z_n$ at the $4(n^2-3)$ $3$-points of $\HK$. The exceptional divisor of $\sigma_n$ is denoted by $G$. We denote again the strict transform of $\mathcal{E}_i$, $\mathcal{E}'_i$, $L_j$, $F_{i,j}$, $M$, $N_i$, $N$ in $Y_n$ by the same symbol.
%$\mathcal{E}_i$, $\mathcal{E}'_i$, $L_j$, $F_{i,j}$, $M$, $N$.
Then we have $$\OO_{Y_n} \Big( \sum_{i=0,1,\zeta,\infty} 3 a_i \mathcal{E}_i + \sum_{i=0,1,\zeta,\infty} 3 a_i \mathcal{E}'_i + \sum_{i=0,1,\zeta,\infty} 3 a_i N_i+ \sum_{i=1}^{8d} 3 b_i L_i  \Big) \simeq \LL^{4p}$$ where $\LL:= \sigma_n^*(\LL_1) \otimes \OO_{Y_n}(-M-3G)$.

As in \cite{EV92}, this is the data to construct a $4p$-th root cover of $Y_n$ branch along the nodal curve $$\sum_{i=0,1,\zeta,\infty} \mathcal{E}_i + \sum_{i=0,1,\zeta,\infty} \mathcal{E}'_i + \sum_{i=0,1,\zeta,\infty} N_i+ \sum_{i=1}^{8d} L_i.$$ We now follow the description in \cite{U08,U10}. Notice that the branch multiplicities are either $3$ or $3(2p-1)$, both coprime to $4p$.

We write $$\sum_j \nu_j A_j :=\sum_{i=0,1,\zeta,\infty} 3 a_i \mathcal{E}_i + \sum_{i=0,1,\zeta,\infty} 3 a_i \mathcal{E}'_i + \sum_{i=0,1,\zeta,\infty} 3 a_i N_i+ \sum_{i=1}^{8d} 3 b_i L_i$$ where the $A_j$ are distinct irreducible curves, and $\nu_j$ are the corresponding multiplicities. Set $A=\sum_j A_j$. The construction follows three steps $$ X_n \xrightarrow{f_3} \text{Spec}_{Y_n} \Big( \bigoplus_{i=0}^{4p-1} {\LL^{(i)}}^{-1} \Big) \xrightarrow{f_2} \text{Spec}_{Y_n} \Big( \bigoplus_{i=0}^{4p-1} {\LL}^{-i} \Big) \xrightarrow{f_1} Y_n $$ where $f_1$ is the $4p$-th root cover defined by the $\OO_{Y_n}$-algebra $\bigoplus_{i=0}^{4p-1} {\LL}^{-i}$ induced by $\LL^{4p} \hookrightarrow \OO_{Y_n}$, the morphism $f_2$ is the normalization, where $$\LL^{(i)}:= \LL^i \otimes \OO_{Y_n} \Bigl( - \sum_j \Bigl[\frac{\nu_j \ i}{4p}\Bigr] A_j \Bigr)$$ for $i \in{\{0,1,\ldots,4p-1 \}}$ and $[x]$ is the integral part of $x$, and $f_3$ is the minimal resolution of the singularities of $\text{Spec}_{Y_n} \Big( \bigoplus_{i=0}^{4p-1} {\LL^{(i)}}^{-1} \Big)$. The surface $X_n$ is a nonsingular irreducible projective surface; cf. \cite[\S1]{U10}.

Let $f \colon X_n \to Y_n$ be the composition of the above morphisms. We have the $\QQ$-numerical equivalence $$ K_{X_n} \equiv f^*\Big( K_{Y_n}
+ \frac{(4p-1)}{4p} A \Big) + \Delta,$$ where $\Delta$ is a $\QQ$-divisor supported on the exceptional divisor of $f_3$.

\begin{remark}
In this construction, we can always arrange the multiplicities $\nu_i$ so that $0<\nu_i<4p$. At the end the corresponding surfaces $X_n$ are isomorphic over $Y_n$; cf. \cite[\S1]{U10}. For the sake of simplicity, in this paper we will not adjust multiplicities.
\label{mult}
\end{remark}

The following are key numbers to compute the Chern invariants.

\begin{definition}
Let $q,m$ be coprime integers such that $0<q<m$.
\begin{itemize}
\item[(1)] The associated
\textit{Hirzebruch-Jung continued fraction} is $$ \frac{m}{q} = e_1 - \frac{1}{e_2 - \frac{1}{\ddots - \frac{1}{e_l}}}$$ We denote its \textit{length} as $l(q,m):=l$.

\item[(2)] The \textit{Dedekind sum} associated to the pair $(q,m)$ is defined as
$$s(q,m):= \sum_{i=1}^{m-1} \Bigl(\Bigl(\frac{i}{m}\Bigr)\Bigr)\Bigl(\Bigl(\frac{iq}{m}\Bigr)\Bigr)$$ where $((x))=x-[x]-\frac{1}{2}$ for any
rational number $x$.
\end{itemize}
\end{definition}

These numbers play a central role for the defect caused by the nodes of $A$ in the Chern numbers of $X_n$. General connections between Dedekind sums and geometry can be found in \cite{HiZa74}.

\begin{remark}
Each of the singularities in $\text{Spec}_{Y_n} \Big( \bigoplus_{i=0}^{4p-1} {\LL^{(i)}}^{-1} \Big)$ is over a node of $A$. Let $P \in A$ be a node in $A_i \cap A_j$. Then the singularity over $P$ is locally isomorphic to the normalization of $$(0,0,0) \in \{z^{4p}=x^{\nu_i}y^{\nu_j} \} \subset \CC_{x,y,z}^3 $$ which is the cyclic quotient singularity $\frac{1}{4p}(1,q)$, where $0<q<4p$ with $\nu_i + q \nu_j \equiv 0$ (mod $4p$). For details on these singularities, we refer to \cite{BHPV04}, III \S 5. Let $R=R_1+\ldots +R_{l(q,4p)}$ be the prime decomposition of the exceptional (reduced) divisor over $P$. Let $\tilde{A}_i, \tilde{A}_j$ be the strict pre-images of $A_i, A_j$ in $X_n$ locally over $P$. Then we have $$f^*(A_i)=4p \tilde{A}_i + \sum_{k=1}^{l(q,4p)} c_k R_k \ \ \ \text{and} \ \ \ \ f^*(A_j)= 4p \tilde{A}_j + \sum_{k=1}^{l(q,4p)} d_k R_k,$$ for certain integers $0 < c_k, d_k < 4p$, such that the discrepancy of $R_k$ is (see for example \cite[\S4.3.2]{U08}, \cite{U10} for precise numbers) $$\text{disc}(R_k)=-1+ \frac{c_k}{4p}+\frac{d_k}{4p}.$$ We recall that disc$(R_k) \in ]-1,0] \cap \QQ$ is the coefficient of $R_k$ in $\Delta$.
\label{discr}
\end{remark}

In our case, we have two types of singularities. We have singularities of type $\frac{1}{4p}(1,4p-1)$ over the nodes in a $A_i \cap A_j$ with $\nu_i=\nu_j$, and of type $\frac{1}{4p}(1,2p+1)$ over the other nodes.

\begin{proposition}
The surface $X_n$ is simply connected.
\label{simply}
\end{proposition}

\begin{proof}
This is the same proof as in \cite[\S8]{U10}, we recall it. Take a general point in $P \in L_1 \subset Y_n$. Through that point we have the trivial pencil of lines from $\PP^2$. Let $\psi \colon Y'_n \to Y_n$ be the blow-up of $P$. We now take the pulled-back $4p$-th root cover $f' \colon X'_n \to Y'_n$, coming from the $f \colon X_n \to Y_n$, so that we have the commutative diagram $$ \xymatrix{X'_n \ar[r]^{f'} \ar[d]^{\psi'} & Y'_n \ar[d]^{\psi} \\ X_n \ar[r]^f & Y_n }$$ where $\psi' \colon X'_n \to X_n$ is the blow-down of a chain of $4p$ $\PP^1$'s, starting with a $(-1)$-curve followed by $4p-1$ $(-2)$-curves. Notice that $X'_n$ has a fibration $h \colon X'_n \to \PP^1$ with connected fibers, sections, and at least one simply connected fiber which comes from the line $L_1$. (Singularities coming from points in $L_1$ are resolved through chains of smooth rational curves.) Then we apply \cite{X91} to conclude that $X'_n$ is simply connected, so is $X_n$. For more details see \cite[Prop.8.3]{U10}.
\end{proof}

\begin{proposition}
The surface $X_n$ is spin. In particular, it is minimal.
\label{spin}
\end{proposition}

\begin{proof}
By Proposition \ref{simply}, the surface $X_n$ is simply connected. Thus we only need to prove that $K_{X_n}$ is numerically $2$-divisible.

Let $\mathcal{E}:=\sum_{i=0,1,\zeta,\infty} \mathcal{E}_i$ and $\mathcal{E}':=\sum_{i=0,1,\zeta,\infty} \mathcal{E}'_i$. We have $$K_{X_n} \equiv f^*\Big(-3L + N + E + 2G + \frac{(4p-1)}{4p}(\mathcal{E}+\mathcal{E}'+N+\mathcal{A}_{8d})\Big) + \Delta$$ where $\Delta$ is a $\QQ$-divisor supported on the exceptional locus of the singularities $\frac{1}{4p}(1,2p+1)$ over the nodes of $A$, which are locally of the type $\{x^3y^{3(2p-1)}=0\}$ in $\sum_j \nu_j A_j$. For the other nodes we have rational double points of type $\frac{1}{4p}(1,4p-1)$, and so they do not contribute to $\Delta$.

Notice that $$\mathcal{L}^{4p} -2 \Big(\mathcal{E}_{0}+N_0+\mathcal{E}'_0 + \mathcal{E}_{1}+\mathcal{E}'_1+N_1+\sum_{i=1}^{4d} L_i \Big) $$ $$  -(6p-4)\Big(\mathcal{E}_{\infty}+N_{\infty}+\mathcal{E}'_{\infty}+\mathcal{E}_{\zeta}+N_{\zeta}+\mathcal{E}'_{\zeta}+\sum_{i=4d+1}^{8d} L_i \Big) \sim \mathcal{E}+\mathcal{E}'+N+\mathcal{A}_{8d}.$$ Then $f^*\big(\frac{(4p-1)}{4p}(\mathcal{E}+\mathcal{E}'+N+\mathcal{A}_{8d})\big)$ is numerically equivalent to $$f^*\Big((4p-1) \mathcal{L}\Big) - 2(4p-1)\Big(\bar{\mathcal{E}}_{0}+\bar{\mathcal{E}}'_0+\bar{N}_0+ \bar{\mathcal{E}}_{1}+\bar{\mathcal{E}'}_1+\bar{N}_1+\sum_{i=1}^{4d} \bar{L}_i \Big)  $$ $$ - (6p-4)(4p-1)\Big(\bar{\mathcal{E}}_{\infty}+\bar{\mathcal{E}}'_{\infty}+\bar{N}_{\infty}+\bar{\mathcal{E}}_{\zeta}+\bar{\mathcal{E}'}_{\zeta}+\bar{N}_{\zeta}+
\sum_{i=4d+1}^{8d} \bar{L}_i \Big) -\bar{\Delta}$$ where $\bar{\Gamma}$ denotes the strict (reduced) transform of a curve $\Gamma$ under $f$, and $\bar{\Delta}$ is $\QQ$-numerically effective and supported in the exceptional divisors from $\frac{1}{4p}(1,4p-1)$, and $\frac{1}{4p}(1,2p+1)$.

Over one $\frac{1}{4p}(1,4p-1)$, we have that $\bar{\Delta}$ is either $(6p-4)(4p-1)R$ or $2(4p-1)R$ according to the multiplicities $\nu_i=\nu_j$ in the corresponding node, where $R$ is the chain of $4p-1$ $(-2)$-curves which resolves $\frac{1}{4p}(1,4p-1)$. In this case $\Delta=0$, so $-\bar{\Delta}+\Delta$ is even.

The singularity $\frac{1}{4p}(1,2p+1)$ is resolved by a chain of three $\PP^1$'s, say $R_1$, $R_2$, $R_3$, such that $R_1^2=R_3^2=-2$ and $R_2^2=-(p+1)$. One can compute that over one $\frac{1}{4p}(1,2p+1)$, we have that $$\bar{\Delta}=\frac{4p-1}{4p} \Big( (10p-2)R_1 + (12p-4)R_2 + (12p^2-2p-2)R_3 \Big)$$ and $-\Delta=\frac{p-1}{2p}R_1 + \frac{p-1}{p}R_2 +\frac{p-1}{2p}R_3$. Therefore $-\bar{\Delta}+\Delta \in \ZZ$ and even.

Thus we only need to prove that $f^*\big(-3L + N + E + 2G +(4p-1)\mathcal{L} \big)$ is divisible by $2$. But $$\mathcal{L} \sim 6p(6\alpha^2+\beta^2)(a_0 F_0 + a_1 F_1 + a_{\infty} F_{\infty} + a_{\zeta} F_{\zeta})-3E-3G-M+6dL$$ so we only need $f^*(-3L+ N + G +M)$ to be divisible by $2$. But $M+3N+3G \sim 9L$, since this is the pull-back of the dual Hesse arrangement.
\end{proof}

The arrangement $A$ has only $2$-points, and its number is $$t_2 = 13824\beta^2\alpha^2 p^4+1152 \beta^4 p^4+4608d\alpha^2 p^2+768d \beta^2 p^2+32d^2-100d.$$ By Proposition \ref{logchern}, its log Chern numbers are $$\bar{c}_1^2= n^4 +2t_2 -40d -48 \ \ \text{and} \ \ \bar{c}_2= \frac{n^4}{3} +t_2 -16d -12.$$

For the Chern numbers, we have $$ c_1^2(X_n) = 4p \, \bar{c}_1^2 - 2\Big(t_2 + 2 \sum_j (g(A_j)-1) \Big) + \frac{1}{4p} \sum_j A_j^2 - \sum_{i<j} c(q_{i,j},4p) A_i \cdot A_j$$ and $$c_2(X_n) = 4p \, \bar{c}_2 - \Big(t_2 + 2 \sum_j (g(A_j)-1) \Big) + \sum_{i<j} l(q_{i,j},4p) A_i \cdot A_j$$ where $0<q_{i,j}<4p$ with $\nu_i + q_{i,j} \nu_j \equiv 0$ (mod $4p$), and $$c(q_{i,j},4p):=12s(q_{i,j},4p)+l(q_{i,j},4p)$$ See \cite[\S4.3]{U08} for proofs which work whenever the multiplicities of the branch divisor (in our case $3$ and $2p-3$ after reduction modulo $4p$) are coprime to the order of the cyclic cover (in our case $4p$).

We also have that $c(4p-1,4p)=\frac{4p-1}{2p}$ and $c(2p+1,4p)=\frac{2p^2+1}{2p}$ \cite[p.345]{U10}, and $l(4p-1,4p)=4p-1$ and $l(2p+1,4p)=3$. Therefore, $$\sum_{i<j} c(q_{i,j},4p) A_i \cdot A_j=\frac{(4p-1)}{2p} t_{2,1}+ \frac{(2p^2+1)}{2p} t_{2,2}$$ and $$\sum_{i<j} l(q_{i,j},4p) A_i \cdot A_j= (4p-1) t_{2,1}+ 3 t_{2,2}$$ where $t_{2,1}$ and $t_{2,2}$ are the number of $2$-points corresponding to the singularities $\frac{1}{4p}(1,4p-1)$ and $\frac{1}{4p}(1,2p+1)$ respectively. Hence $$t_{2,1}=384 \beta^4 p^4+4608 \alpha^2 \beta^2 p^4+2304 d \alpha^2 p^2-52 d+384 d \beta^2 p^2+16 d^2$$ and $$t_{2,2}=768\beta^4 p^4+ 9216 \alpha^2 \beta^2 p^4 +2304 d \alpha^2 p^2-48d+384 d \beta^2 p^2+16 d^2.$$

By plugging in the formulas above, we obtain that $$\text{lim}_{p \to \infty} \ \frac{c_1^2}{c_2} = \frac{108x^4+132x^2+11}{36x^4+96x^2+8}=:\lambda(x)$$ where $x=\alpha/\beta$. Notice that $[2,3] \subset \lambda\big([0,\infty^+]\big)=[1.375,3]$.

\begin{theorem}
For any number $r \in [1.375,3]$, there are spin simply connected minimal surfaces of general type $X$ with $c_1^2(X)/c_2(X)$ arbitrarily close to $r$.
\label{main}
\end{theorem}

\begin{proof}
Let $\epsilon>0$. Then there exists a positive rational number $\alpha/\beta$ such that $|\lambda(\alpha/\beta)-r|<\epsilon$. We take those $\alpha, \beta$ for the construction above. The surfaces $X_n$ are spin by Proposition \ref{spin}, and simply connected by Proposition \ref{simply}. Using Enriques' classification of surfaces, it is clear that the surfaces $X_n$ are of general type for $p>>0$, since $c_1^2(X_n)$ and $c_2(X_n)$ approach $+\infty$ as $p$ tends to $+\infty$.
\end{proof}

\begin{example}
For parameters $\alpha=1$, $\beta=0$, $d=1$, and $p=5$, we obtain a spin simply connected surface with $c_1^2=2^6\cdot 3^2\cdot 455393$ and $c_2=2^5\cdot3^3\cdot 103997$, thus with Chern slope around $2.919$. For $p=100003$, we get a Chern slope close to $3$ up to $2\cdot 10^{-10}$.
\end{example}

%----------------------------------------------------------------------------------------------------------------------------------------------
\section{Non-spin surfaces and other fundamental groups} \label{other}

For non-spin surfaces, we consider the same set up as in Section \ref{spin} but for a $p$-th root cover. We keep same notation as in the previous section up to some explicitly redefined numbers and symbols.

Let $p \geq 5$ be a prime number, and let $\alpha>0, \beta > 0$ be integers. Let $n=3 \alpha p$. As before, we consider the arrangement ${\HH}'_n$ in $H$. Let $\mathcal{E}'_i$ be $\beta^2 p^2$ general fibers of $\pi'_i$, and let $\AR_{2d}=L_1+ \ldots + L_{2d}$ be the strict transform of an arrangement of $2d$ general lines in $\PP^2$, where $3 \leq 2d \leq p$. We define $a_0=a_1=b_i=1$ for $1\leq i \leq d$, and $a_{\infty}=a_{\zeta}=b_i=p-1$ for $d+1 \leq i \leq 2d$. Then $$ \OO_H \Big( \sum_{i=0,1,\zeta,\infty} 3 a_i \mathcal{E}_i + \sum_{i=0,1,\zeta,\infty} 3 a_i \mathcal{E}'_i + \sum_{i=0,1,\zeta,\infty} a_i (F_{i,1}+F_{i,2}+F_{i,3})+ \sum_{i=1}^{2d} 3 b_i L_i \Big)$$ is isomorphic to $\LL_0^{p}$ where $$\LL_0:= \OO_H \Big( 3p(3\alpha^2+\beta^2)\big(\sum_{i=0,1,\zeta,\infty} a_i F_i \big)+3dL \Big).$$

For each $i$, we denote the strict transform of $\mathcal{E}_i$, $\mathcal{E}'_i$, $L_j$, $F_{i,j}$ in $Z_n$ by the same symbol.
%$\mathcal{E}_i$, $\mathcal{E}'_i$, $L_j$, $F_{i,j}$.
Then $$\OO_{Z_n} \big( \sum_{i=0,1,\zeta,\infty} 3 a_i \mathcal{E}_i + \sum_{i=0,1,\zeta,\infty} 3 a_i \mathcal{E}'_i + \sum_{i=0,1,\zeta,\infty} a_i (F_{i,1}+F_{i,2}+F_{i,3})+ \sum_{i=1}^{2d} 3 b_i L_i  \big)$$ is $\LL_1^{p}$ where  $\LL_1:= \varphi_n^*(\LL_0) \otimes \OO_{Z_n}(-6E)$, and $E$ is the exceptional divisor of $\varphi_n$. Again, we denote the strict transform of $\mathcal{E}_i$, $\mathcal{E}'_i$, $L_j$, $F_{i,j}$, $M$, $N_i$, $N$ in $Y_n$ by the same symbol.
%$\mathcal{E}_i$, $\mathcal{E}'_i$, $L_j$, $F_{i,j}$, $M$, $N$.
Then we have $$\OO_{Y_n} \Big( \sum_{i=0,1,\zeta,\infty} 3 a_i \mathcal{E}_i + \sum_{i=0,1,\zeta,\infty} 3 a_i \mathcal{E}'_i + \sum_{i=0,1,\zeta,\infty} 3 a_i N_i+ \sum_{i=1}^{2d} 3 b_i L_i  \Big) \simeq \LL^{p}$$ where $\LL:= \sigma_n^*(\LL_1) \otimes \OO_{Y_n}(-2M-6G)$. With this data, we construct a $p$-th root cover of $Y_n$ branch along $$A:= \sum_{i=0,1,\zeta,\infty} \mathcal{E}_i + \sum_{i=0,1,\zeta,\infty} \mathcal{E}'_i + \sum_{i=0,1,\zeta,\infty} N_i+ \sum_{i=1}^{2d} L_i.$$ Let $f \colon X_p \to Y_n$ be the corresponding morphism for the $p$-th root cover, as in the previous section.

\begin{proposition}
The surfaces $X_p$ are simply connected.
\label{simply2}
\label{min}
\end{proposition}

\begin{proof}
This is the same as Proposition \ref{simply} using $\AR_{2d}$.
\end{proof}

Let us write $$\sum_j \nu_j A_j= \sum_{i=0,1,\zeta,\infty} 3 a_i \mathcal{E}_i + \sum_{i=0,1,\zeta,\infty} 3 a_i \mathcal{E}'_i + \sum_{i=0,1,\zeta,\infty} 3 a_i N_i+ \sum_{i=1}^{2d} 3 b_i L_i$$ where $A_j$ are the irreducible curves in $A$. Hence $\nu_j$ is equal to either $3a_i$ or $3b_k$ for some $i,k$. The absence of $\PP^1$'s coming from $k$-points in the branch divisor allows us to prove the following \cite[IV \S4.3.2]{U08}.

\begin{proposition}
The surfaces $X_p$ are minimal.
\label{min}
\end{proposition}

\begin{proof}
The argument will be two parts. First we show that $K_{X_p}$ is $\QQ$-numerically equal to an effective divisor supported on the pre-image of $A+E+G$ and the exceptional divisor of $f_3$ (see construction of $f$ in the previous section). The second part is to show that no curve in the pre-image of $A+E+G$ is a $(-1)$-curve.

We know that $K_{X_p} \equiv f^*\Big(K_{Y_n} + \frac{(p-1)}{p} A \Big) + \Delta$ where $\Delta$ is supported on the exceptional divisor of $f_3$. We have that $K_{Y_n} \sim -3L + N + E + 2G$, and $$-\frac{1}{4} (F_0 + F_1 +F_{\infty} + F_{\zeta}) \equiv -3L + \frac{3}{4}N + \frac{3}{4}G.$$ If $u= \frac{p-1}{p}$ and $v=\beta^2p^2$, then $K_{X_p}$ is numerically $\QQ$-equivalent to $$f^*\Big( \sum_{i=0,1,\zeta,\infty} u \mathcal{E}_i + \sum_{i=0,1,\zeta,\infty} \big(u-\frac{1}{4v} \big) \mathcal{E}'_i + \big(u+\frac{1}{4} \big) N + u \AR_{2d} + E + \frac{5}{4}G \Big) + \Delta.$$ On the other hand, for a node $P \in A_i \cap A_j$ we locally have $$f^*(A_i+A_j)=p \tilde{A}_i + p \tilde{A}_j + \sum_{k=1}^l (c_k+d_k) R_k $$ as in Remark \ref{discr}, where $R_1+ \ldots+ R_l$ is the exceptional divisor above $P$, and discr$(R_k)=-1+\frac{c_k}{p} + \frac{d_k}{p}$. We clearly have $$xc_k + yd_k -1+\frac{c_k}{p} + \frac{d_k}{p} >0$$ where $x,y \in \{u,u-\frac{1}{4v},u+\frac{1}{4}\}$. Therefore, $K_{X_p}$ is $\QQ$-numerically effective with support equal to the pre-image of $A+E+G$ and the exceptional divisor of $f_3$.

Notice that in this support the only $\PP^1$'s are the ones coming from $N$, lines in $\AR_{2d}$, and $\Delta$. The ones in $\Delta$ are not $(-1)$-curves. For the ones in $N$, we have $\PP^1$'s intersecting $\sum_{i=0,1,\zeta,\infty} \mathcal{E}'_i$ at $3v$ points. Similarly with the $\PP^1$'s coming from lines $L_j$, we have intersection at $4(n^2-3)+12v$ points with $\sum_{i=0,1,\zeta,\infty} \mathcal{E}_i + \sum_{i=0,1,\zeta,\infty} \mathcal{E}'_i$. Using the notation from Remark \ref{discr}, for the strict pre-image $\tilde{A}_j$ of $A_j$ under $f$, we have (see \cite[IV \S4.3.2]{U08}) $$\tilde{A}_j^2= \frac{1}{p} \big( A_j^2 - \sum_{i \neq j} q_{i,j} A_i \cdot A_j \big) < -1.$$ Therefore, we have no $(-1)$-curves in the support.
\end{proof}

The arrangement $A$ has only $2$-points, and its number is $$t_2 = 108 \alpha^2 \beta^2 p^4 +18 \beta^4 p^4+ 72d \alpha^2 p^2-25d+24d \beta^2p^2+2d^2.$$ By Proposition \ref{logchern}, its log Chern numbers are $$\bar{c}_1^2= n^4 +2t_2 -10d -48 \ \ \text{and} \ \ \bar{c}_2= \frac{n^4}{3} +t_2 -4d -12.$$

We also have that $c(p-1,p)=\frac{2p-2}{p}$ and $c(1,p)=\frac{p^2-2p+2}{p}$ \cite[p.345]{U10}, and $l(p-1,p)=p-1$ and $l(1,p)=1$. Therefore, $$\sum_{i<j} c(q_{i,j},4p) A_i \cdot A_j=\frac{(2p-2)}{p} t_{2,1}+ \frac{(p^2-2p+2)}{p} t_{2,2}$$ and $$\sum_{i<j} l(q_{i,j},4p) A_i \cdot A_j= (p-1) t_{2,1}+ t_{2,2}$$ where $t_{2,1}$ and $t_{2,2}$ are the number of $2$-points corresponding to the singularities $\frac{1}{p}(1,p-1)$ and $\frac{1}{p}(1,1)$ respectively. Hence $$t_{2,1}=6 \beta^4 p^4+36 \alpha^2 \beta^2 p^4+36 d \alpha^2 p^2-13 d+12 d \beta^2 p^2+ d^2$$ and $$t_{2,2}=12\beta^4 p^4+ 72 \alpha^2 \beta^2 p^4 +36 d \alpha^2 p^2-12 d+12 d \beta^2 p^2+ d^2.$$

By plugging in the formulas for Chern numbers in Section \ref{spin}, we obtain that $$\text{lim}_{p \to \infty} \ \frac{c_1^2}{c_2} = \frac{27x^4+48x^2+8}{9x^4+48x^2+8}=:\lambda(x)$$ where $x=\alpha/\beta$. Notice that $[2,3] \subset \lambda\big([0,\infty^+]\big) =[1,3]$.

\begin{theorem}
For any number $r \in [1,3]$, there are non-spin simply connected minimal surfaces of general type $X$ with $c_1^2(X)/c_2(X)$ arbitrarily close to $r$.
\label{main}
\end{theorem}

\begin{proof}
Let $\epsilon>0$. Then there exists a positive rational number $\alpha/\beta$ such that $|\lambda(\alpha / \beta)-r|<\epsilon$. We take those $\alpha, \beta$ for the construction above. The surfaces $X_p$ are minimal by Proposition \ref{min}, simply connected by Proposition \ref{simply2}, and non-spin because of the degree of $f$. The surfaces $X_p$ are of general type for $p>>0$.
\end{proof}

Let $q>0$ be an integer, and let $\rho \colon C_q \to \PP^1$ be a $(q+1)$-cyclic cover completely branch at four general points. Hence the genus of $C_q$ is $q$.

\begin{corollary}
For any number $r \in [1,3]$, there are minimal surfaces of general type $X$ with $c_1^2(X)/c_2(X)$ arbitrarily close to $r$, and $\pi_1(X) \simeq \pi_1(C_q)$.
\label{fund}
\end{corollary}

\begin{proof}
Consider the $X_p$ in Theorem \ref{main}. Take a general point in $P \in L_1 \subset Y_n$. Through that point we have the trivial pencil of lines from $\PP^2$. As we did in Proposition \ref{simply}, let $\psi \colon Y'_n \to Y_n$ be the blow-up of $P$. We now take the pulled-back $p$-th root cover $f' \colon X'_p \to Y'_n$, coming from the $f \colon X_p \to Y_n$, so that we have the commutative diagram $$ \xymatrix{X'_p \ar[r]^{f'} \ar[d]^{\psi'} & Y'_n \ar[d]^{\psi} \\ X_p \ar[r]^f & Y_n }$$ where $\psi' \colon X'_p \to X_p$ is the blow-down of a chain of $p$ $\PP^1$'s, starting with a $(-1)$-curve followed by $p-1$ $(-2)$-curves. Notice that $X'_p$ has a fibration $h \colon X'_p \to \PP^1$ with connected fibers, sections, and at least one simply connected fiber which comes from the line $L_1$. The genus $g$ of the general fiber of $h$ satisfies the Riemann-Hurwitz formula $2g-2=p(-2) + (p-1)(4(n^2-3)+12\beta^2 p^2+2d)$. Consider the base change fibration $$\xymatrix{X \ar[r]^{\rho'} \ar[d]^{h'} & X'_p \ar[d]^{h} \\ C_q \ar[r]^{\rho} & \PP^1 }$$ where the surface $X$ is nonsingular projective of general type such that $$\frac{c_1^2(X)}{c_2(X)} = \frac{(q+1)(c_1^2(X_p)-p) + 16q(g-1)}{(q+1)(c_2(X_p)+p)+8q(g-1)},$$ and so, since the highest power of $p$ in $g-1$ is $p^3$, $\frac{c_1^2(X)}{c_2(X)}$ approaches $\lambda(\alpha / \beta)$ as $p$ tends to infinity. For minimality we notice that $$K_X \equiv \rho'^*\Big(\psi'^*(K_{X_p}) + \sum_{i=1}^p i S_i + \frac{4q}{q+1} F \Big)$$ where $\sum_{i=1}^p S_i$ is the chain of $\PP^1$ explained above, with $S_{p}^2=-1$, and $F$ is a general fiber of $h$. We know that $K_{X_p}$ is $\QQ$-effective. Therefore $K_X$ is $\QQ$-effective. One can verify that its support does not contain $(-1)$-curves, and so $X$ is minimal. Also, notice that $h'$ has sections, for instance the pre-image of $S_p$, and it has some simply connected fibers from the pre-image of the one in $h$. Hence, by \cite{X91}, we have that $\pi_1(X) \simeq \pi_1(C_q)$.
\end{proof}

\begin{remark}
The arrangement of either $8d$ or $2d$ general lines in the constructions above gives not only simply connectedness but also shows that the constructed surfaces have unbounded deformations. It also exhibits some freedom for the constructions, where we can add branching curves in such a way that they either do not depend on $p$ or do depend on $p$ but affecting Chern numbers in a suitable power of $p$.
\end{remark}

%\begin{remark}
%Same results can be obtained using the special arrangements of elliptic curves in \cite{KH05} instead of Hirzebruch's ones.
%\end{remark}

%----------------------------------------------------------------------------------------------------------------------------------------------

\vspace{0.3cm}

{\small D\'epartement de Math\'ematiques,

Universit\'e de Poitiers,

Poitiers, France.}

\vspace{0.3cm}

{\small Facultad de Matem\'aticas,

Pontificia Universidad
Cat\'olica de Chile,

Santiago, Chile.}
\vspace{0.3cm}

\end{document}